\documentclass[12pt,conference]{IEEEtran}
\ifCLASSOPTIONcompsoc
  \usepackage[nocompress]{cite}
\else
  \usepackage{cite}
\fi
\usepackage{amsmath,amssymb,url}
\usepackage{epic}
\usepackage{eepic}
\newsavebox{\wwide}
\newcommand{\wwidehat}[1]{\sbox{\wwide}{$#1$}
\ifdim\wd\wwide < 1.1 em \widehat{#1} \else
\setlength
{\unitlength}{0.01\wd\wwide}\overset
{\begin{picture}(100,6)
\path(0,0)(50,6)(100,0)
\end{picture}}{#1}\fi}

\newcommand{\wwidetilde}[1]{\sbox{\wwide}{$#1$}
\ifdim\wd\wwide < 1.1 em \widetilde{#1} \else
\setlength
{\unitlength}{0.01\wd\wwide}\overset
{\begin{picture}(100,6)
\path(0,0)(33,6)(45,6)(55,0)(67,0)(100,6)
\end{picture}}{#1}\fi}

\ifx\pdfoutput\@undefined\usepackage[usenames,dvips]{color}
\else\usepackage[usenames,dvipsnames]{color}
\IfFileExists{pdfcolmk.sty}{\usepackage{pdfcolmk}}{} 
\fi

\definecolor{ChadDarkBlue}{rgb}{.1,0,.2}  
\definecolor{ChadBlue}{rgb}{.1,.1,.5}  
\definecolor{ChadRoyal}{rgb}{.2,.2,.8}  
\definecolor{ChadGreen}{rgb}{0,.4,0}    % Dark Green
\definecolor{ChadRed}{rgb}{.5,0,.5}  % purple

\def\zmena#1{{#1}} 
\def\autori#1{{#1}} 

\usepackage{amsthm}
\usepackage{epic,eepic}
\usepackage[mathscr]{euscript}
\usepackage{enumerate}

\usepackage{graphicx}

\usepackage{lineno}
\usepackage{comment}
\def\smallskip{\vskip\smallskipamount}
\def\medskip{\vskip\medskipamount}
\def\bigskip{\vskip\bigskipamount}

\numberwithin{equation}{section}

\theoremstyle{plain}
\newtheorem{theorem}{Theorem}[section]
\newtheorem{lemma}[theorem]{Lemma}

\newtheorem{corollary}[theorem]{Corollary}

\theoremstyle{definition}
\newtheorem{definition}[theorem]{Definition}

\newtheorem{remark}[theorem]{Remark}

\newtheorem{statement}[theorem]{Statement}

\def\operatorname#1{{\mathop{\rm #1}}}
\newcommand{\ord}{\operatorname{ord}}

\def\comp{\leftrightarrow}

\newcommand\implik{{\ \Longrightarrow\ }}

\newcounter{ok}
{\end{list}}

\newcounter{aok}
{\end{list}}

\def\go#1;#2;#3 {\vbox to0pt{\kern-#3\hbox{\kern#2 #1}\vss}\nointerlineskip}

\def\moplus{\oplus}
\def\mominus{\ominus}

\newcommand{\Mea}{{\text{\rm{}M}}}

\newcommand{\Tea}{{\text{\rm{}T}}}

\newcommand{\C}{\text{\rm{}C}}
\newcommand{\Sh}{\text{\rm{}S}}

\newcommand{\itrm}[1]{\item[{\rm {#1}}]}
\newcommand{\Cen}{\text{\rm{}C}}

 \renewcommand{\le}{\leqslant}

\begin{document}
\title{Triple Representation Theorem for homogeneous effect algebras}
%{Department of Mathematics and Statistics,
%			Faculty of Science, Masaryk University,
%			{Kotl\'a\v r{}sk\' a\ 2}, CZ-611~37~Brno, Czech Republic}
\author{\IEEEauthorblockN{Josef~Niederle and Jan~Paseka}
\IEEEauthorblockA{Department of Mathematics and Statistics\\
		Faculty of Science, Masaryk University\\
		{Kotl\'a\v r{}sk\' a\ 2}, CZ-611~37~Brno, Czech Republic\\
        E-mail: niederle,paseka@math.muni.cz\\[1cm]
{\it Dedicated to Ivo G. Rosenberg on the occasion of his 77th birthday} \\{\it in appreciation of his contributions to the fields of multiple-valued logic}}}

%\inst{1}\email{niederle@math.muni.cz} 
%\and  \inst{1}\email{paseka@math.muni.cz}} %etc.

%\IEEEaftertitletext{This paper is dedicated to Lawrence Markus %Emeritus, Regents Professor of Mathematics, University of Minnesota% in Minneapolis, for his many and manifold contributions to mathematics}

%\IEEEcompsocitemizethanks{\IEEEcompsocthanksitem 
%J. Niederle and J. Paseka are with the Department of Mathematics and Statistics,Faculty of Science,Masaryk University,
%{Kotl\'a\v r{}sk\' a\ 2},
%CZ-611~37~Brno, Czech Republic.\protect\\
%E-mail: niederle,paseka@math.muni.cz}        	

\markboth{Triple Representation Theorem for homogeneous effect algebras}%
{Triple Representation Theorem for homogeneous effect algebras}

%\def\received{Received \today; In final form }
%

%\subjclass{MSC 03G12 \and 06D35 \and 06F25 \and 81P10}

%% Keywords and phrases


\IEEEcompsoctitleabstractindextext{%
\begin{abstract}
%\boldmath
The aim of our paper is to prove the Triple Representation Theorem,
which was established 
by Jen\v{c}a in the setting of complete lattice effect 
algebras, for a special class of 
homogeneous effect algebras, namely TRT-effect algebras. This class 
includes complete  lattice effect algebras, 
sharply dominating Archimedean 
atomic  lattice effect algebras and 
homogeneous orthocomplete effect algebras.
\end{abstract}
% not want either math or citations to appear in the abstract.

}

\maketitle  
%\end{document}

% is not selected <OR> if conference mode is selected - because compsoc
% conference papers position the abstract like regular (non-compsoc)
\IEEEdisplaynotcompsoctitleabstractindextext

% \fi
%\IEEEpeerreviewmaketitle

%\end{document}
{\section*{Introduction}}

\label{intro}

\IEEEPARstart{T}{wo} equivalent quantum structures, D-posets and effect algebras were
introduced in the nineties of the twentieth century. These
were considered as "unsharp" generalizations of the
structures which arise in quantum mechanics, in particular, of orthomodular
lattices and MV-algebras. Effect algebras aim to
describe "unsharp" event structures in quantum mechanics
in the language of algebra.

Effect algebras are fundamental in investigations of fuzzy probability
theory too. In the fuzzy probability frame, the elements of an effect algebra
represent fuzzy events which are used to construct fuzzy random variables.

The aim of our paper is to prove the Triple Representation Theorem,
which was established 
by Jen\v{c}a in \cite{jenca} in the setting of complete lattice effect 
algebras, for a special class of 
homogeneous effect algebras, namely TRT-effect algebras. This class 
includes complete  lattice effect algebras, 
sharply dominating  Archimedean 
atomic  lattice effect algebras (see \cite{niepa}) and 
homogeneous orthocomplete effect algebras (see \cite{niepa3}).

%\end{document}

\medskip

\section{{Preliminaries} and basic facts}

\zmena{Effect algebras were introduced by Foulis and 
 Bennett (see \cite{FoBe}) 
for modelling unsharp measurements in a Hilbert space. In this case the 
set $\EuScript E(H)$ of effects is the set of all self-adjoint operators $A$ on 
a Hilbert space $H$ between the null operator $0$ and the identity 
operator $1$ and endowed with the partial operation $+$ defined 
iff  $A+B$ is in $\EuScript E(H)$, where $+$ is the usual operator sum. }

\zmena{In general form, an effect algebra is in fact a partial algebra 
with one partial binary operation and two unary operations satisfying 
the following axioms due to Foulis and 
 Bennett.}

\begin{definition}\label{def:EA}{\autori{{\rm{}\cite{FoBe}  }}
{\rm A partial algebra $(E;\oplus,0,1)$ is called an {\em effect algebra} if
$0$, $1$ are two distinct elements, 
called the {\em zero} and the {\em unit}  element,  and $\oplus$ is a partially
defined binary operation called the {\em orthosummation} 
on $E$ which satisfies the following
conditions for any $x,y,z\in E$:
\begin{description}
\item[\rm(Ei)\phantom{ii}] $x\oplus y=y\oplus x$ if $x\oplus y$ is defined,
\item[\rm(Eii)\phantom{i}] $(x\oplus y)\oplus z=x\oplus(y\oplus z)$  if one
side is defined,
\item[\rm(Eiii)] for every $x\in E$ there exists a unique $y\in
E$ such that $x\oplus y=1$ (we put $x'=y$),
\item[\rm(Eiv)\phantom{i}] if $1\oplus x$ is defined then $x=0$.
\end{description}
}%
{\rm{}$(E;\oplus,0,1)$ is  called an {\em orthoalgebra} if 
$x\oplus x$ exists implies that $x= 0$  (see \cite{grerut}).}}
\end{definition}

We often denote the effect algebra $(E;\oplus,0,1)$ briefly by
$E$. On every effect algebra $E$  a partial order
$\le$  and a partial binary operation $\ominus$ can be 
introduced as follows:
%\noindent{}
\begin{center}
$x\le y$ \mbox{ and } {\autori{$y\ominus x=z$}} \mbox{ iff }$x\oplus z$
\mbox{ is defined and }$x\oplus z=y$\,.
\end{center}

If $E$ with the defined partial order is a lattice (a complete
lattice) then $(E;\oplus,0,1)$ is called a {\em lattice effect
algebra} ({\em a complete lattice effect algebra}).

Mappings from one effect algebra to
another one that preserve units and orthosums are called 
{\em morphisms of effect algebras}, and bijective
morphisms of effect algebras having inverses that are 
morphisms of effect algebras are called 
{\em isomorphisms of effect algebras}.

\begin{definition}\label{subef}{\rm
Let $E$ be an  effect algebra.
Then $Q\subseteq E$ is called a {\em sub-effect algebra} of  $E$ if 
\begin{enumerate}
\item[(i)] $1\in Q$
\item[(ii)] if out of elements $x,y,z\in E$ with $x\oplus y=z$
two are in $Q$, then $x,y,z\in Q$.
\end{enumerate}
}
\end{definition}

Note that a sub-effect algebra $Q$ 
 of an  effect algebra $E$ 
 with inherited operation 
$\oplus$ is an  effect algebra  
in its own right.

\begin{definition}%\advance\parindent by 5pt
%\medskip 

\rm
%\begin{description}
{(1)}: A {\em generalized effect algebra} ($E$; $\oplus$, 0) is a set $E$ with element $0\in E$ and partial
binary operation $\oplus$ satisfying, for any $x,y,z\in E$, the conditions
\medskip

{\advance\parindent by 5pt
\begin{tabular}{@{}c@{}l}
&\begin{minipage}{0.47679454\textwidth}
\begin{enumerate}
\item[(GE1)] $x\oplus y = y\oplus x$ if one side is defined,

\item[(GE2)] $(x\oplus y)\oplus z=x\oplus(y\oplus z)$ if one side is
defined,

\item[(GE3)] if $x\oplus y=x\oplus z$ then $y=z$,

\item[(GE4)] if $x\oplus y=0$ then $x=y=0$,

\item[(GE5)] $x\oplus 0=x$ for all $x\in E$.
\end{enumerate}
\end{minipage}
\end{tabular}}
\medskip

\begin{description}
\item[(2)] A binary relation $\le$ (being a partial order) and 
a partial binary operation $\ominus$ on $E$ can be defined by:
%\noindent{}
\begin{center}
$x\le y$ \mbox{ and } {{$y\ominus x=z$}} \mbox{ iff }$x\oplus z$
\mbox{ is defined and }$x\oplus z=y$\,.
\end{center} 
\item[(3)] A nonempty subset $Q\subseteq E$ is called a {\em sub-generalized effect algebra} 
of $E$ if out of elements $x,y,z\in E$ with $x\oplus y=z$ at least two are in $Q$ then $x,y,z\in
Q$.   
\end{description}

\end{definition}

%\Longleftarrow 

Every sub-generalized effect algebra of $E$ is a generalized effect algebra in its own right. Every effect algebra is a generalized effect algebra.

For an element $x$ of a generalized effect algebra $E$ we write
$\ord(x)=\infty$ if $nx=x\oplus x\oplus\dots\oplus x$ ($n$-times)
exists for every positive integer $n$ and we write $\ord(x)=n_x$
if $n_x$ is the greatest positive integer such that $n_xx$
exists in $E$.  A generalized effect algebra $E$ is {\em Archimedean} if
$\ord(x)<\infty$ for all $x\in E$.

A minimal nonzero element of a generalized effect algebra  $E$
is called an {\em atom}  and $E$ is
called {\em atomic} if below every nonzero element of 
$E$ there is an atom.

\begin{definition}
\rm
We say that a finite system $F=(x_k)_{k=1}^n$ of not necessarily
different elements of a generalized effect algebra $E$ is
\zmena{\em orthogonal} if $x_1\oplus x_2\oplus \cdots\oplus
x_n$ (written $\bigoplus\limits_{k=1}^n x_k$ or $\bigoplus F$) exists
in $E$. Here we define $x_1\oplus x_2\oplus \cdots\oplus x_n=
(x_1\oplus x_2\oplus \cdots\oplus x_{n-1})\oplus x_n$ supposing
that $\bigoplus\limits_{k=1}^{n-1}x_k$ is defined and
$(\bigoplus\limits_{k=1}^{n-1}x_k)\oplus x_n$ exists. We also define 
$\bigoplus \emptyset=0$.
An arbitrary system
$G=(x_{\kappa})_{\kappa\in H}$ of not necessarily different
elements of $E$ is called \zmena{\em orthogonal} if $\bigoplus K$
exists for every finite $K\subseteq G$. We say that for 
an \zmena{orthogonal} 
system $G=(x_{\kappa})_{\kappa\in H}$ the
element $\bigoplus G$ exists iff
$\bigvee\{\bigoplus K
%\{\bigoplus\limits_{\kappa\in K}a_{\kappa}
\mid
K\subseteq G$ is finite$\}$ exists in $E$ and then we put
$\bigoplus G=\bigvee\{\bigoplus K\mid K\subseteq G$ is
finite$\}$. We say that $\bigoplus G$ is the {\em orthogonal sum} 
of $G$ and $G$ is {\em orthosummable}. (Here we write $G_1\subseteq G$ iff there is
$H_1\subseteq H$ such that $G_1=(x_{\kappa})_{\kappa\in
H_1}$).
{%\color{magenta}
We denote $G^\oplus:=\{\bigoplus K\mid K\subseteq G$ is   
finite$\}$. $G$ is called {\em bounded} if there is an upper bound 
of $G^\oplus$.

A generalized effect algebra
$E$ is called \emph{orthocomplete} if  every bounded orthogonal 
system  is orthosummable.}

A generalized effect algebra $E$ has the {\em maximality property} 
if $\{ u, v \}$ has a maximal
lower bound $w$ for every $u, v \in E$.
\end{definition}

{\renewcommand{\labelenumi}{{\normalfont  (\roman{enumi})}}
\begin{definition}
\rm
An element $x$ of an effect algebra $E$ is called  
\begin{enumerate}%[(iiiii)]
%\settowidth{\leftmargin}{(iiiii)}
%\settowidth{\labelwidth}{(iii)}
%\settowidth{\itemindent}{(ii)}
\item  \emph{sharp} if $x\wedge x'=0$. The set 
$\Sh(E)=\{x\in E \mid x\wedge x'=0\}$ is called the \emph{set of all sharp elements} 
of $E$ (see \cite{gudder1}).
\item  \emph{principal}, if $y\oplus z\leq x$ for every $y, z\in E$ 
such that $y, z \leq x$ and $y\oplus z$ exists.
\item  \emph{central}, if $x$ and $x'$ are principal and,
for every $y \in E$ there are $y_1, y_2\in E$ such that 
$y_1\leq x, y_2\leq x'$, and $y=y_1\oplus y_2$  (see \cite{GrFoPu}). 
The \emph{center} 
$\C(E)$ of $E$ is the set of all central elements of $E$.
\end{enumerate}
\end{definition}

If $x\in E$ is a principal element, then $x$ is sharp and the interval 
$[0, x]$ is an effect algebra with the greatest element $x$ and the partial 
operation given by restriction of $\oplus$ to $[0, x]$. 

\begin{statement}{\rm \cite[Theorem 5.4]{GrFoPu}}\label{grfopu}
The center $\C(E)$ of an effect algebra $E$ is a sub-effect 
algebra of $E$ and forms a Boolean algebra. For every
central element $x$ of $E$, $y=(y\wedge x)\oplus (y\wedge x')$ for all 
$y\in E$.  If $x, y\in \C(E)$ are orthogonal, we have 
$x\vee y =  x\oplus y$  and $x\wedge y =  0$. 
\end{statement}

{\renewcommand{\labelenumi}{{\normalfont  (\roman{enumi})}}
\begin{statement}{\rm\cite[Lemma   3.1.]{jencapul}}\label{gejzapulm} Let $E$ be an effect algebra, 
$x, y\in E$ and $c, d\in \C(E)$. Then:
\begin{enumerate}
%\settowidth{\leftmargin}{(iiiii)}
%\settowidth{\labelwidth}{(iii)}
%\settowidth{\itemindent}{(ii)}
\item If $x\oplus y$ exists then  $c\wedge (x\oplus y)=(c\wedge x)\oplus (c\wedge y)$.
\item If $c\oplus d$ exists then  $x\wedge (c\oplus d)=(x\wedge c)\oplus (x\wedge d)$.  
\end{enumerate}
\end{statement}

%\end{document}

\begin{definition}{\,\rm{}(\cite{jenca2001,jenca}\,)} 
\rm
A subset $M$ of a generalized effect algebra $E$ is called 
\emph{internally compatible} (\emph{compatible}) if for every finite 
subset $M_F$ of $M$  there is a finite orthogonal family $(x_1, \dots, x_n)$ 
of elements from $M$ ($E$) such that for every $m\in  M_F$ 
there is a set $A_F\subseteq \{ 1, \dots, n \}$ with 
$m =\bigoplus_{i\in A_F} x_i$. If $\{ x, y \}$ is a compatible set, 
we write $x\comp y$ (see \cite{jenca, Kop2}).
\end{definition}

\section{Homogeneous  effect algebras}
\label{orto}

\begin{definition}\label{rdp}{\rm
An effect algebra $E$  satisfies the \emph{Riesz decomposition property}
(or RDP) if, for all $u,v_1, v_2\in E$ such that
$u\leq v_1\oplus v_2$, there are $u_1, u_2$ such that 
$u_1\leq v_1, u_2\leq v_2$ and $u=u_1\oplus u_2$. 
\\
%\begin{color}{magenta}
A lattice effect algebra in which RDP holds is called an \emph{MV-effect
algebra}.
%\end{color}
\\
An effect algebra $E$ is called \emph{homogeneous} if, 
for all $u,v_1,v_2\in E$ such that $u\leq v_1\oplus v_2\leq u'$, 
there are $u_1,u_2$ such that $u_1\leq v_1$, $u_2\leq v_2$ 
and $u=u_1\oplus u_2$ (see \cite{jenca2001}). }
\end{definition}

\begin{statement}\label{gejzasum}{\rm\cite[Proposition 2]{jenca}}
\nopagebreak
\begin{enumerate}
\item[{\rm (i)}]    Every orthoalgebra is homogeneous.
\itrm{(ii)}    Every lattice effect algebra is homogeneous.
\itrm{(iii)}    An effect algebra  $E$ has the Riesz decomposition 
property if and only if  $E$ is homogeneous and compatible.

Let $E$ be a homogeneous effect algebra.
\itrm{(iv)}    A subset $B$ of $E$ is a maximal sub-effect 
algebra of $E$ with the Riesz decomposition property (such  
$B$ is called a {\em block} of  $E$) if and only if  $B$ 
is a maximal internally compatible subset of $E$ containing $1$.
\itrm{(v)} Every finite compatible subset of $E$ is a subset of 
some block. This implies that every homogeneous effect algebra 
is a union of its blocks.
\itrm{(vi)}    $\Sh(E)$ is a sub-effect algebra of $E$.
\itrm{(vii)}    For every block $B$, $\C(B) = \Sh(E)\cap B$.
\itrm{(viii)}    Let $x\in B$, where $B$ is a block of $E$. 
Then $\{ y \in E \mid  y \leq x\ \text{and}\ y \leq x'\}\subseteq  B$.
\end{enumerate}
\end{statement}

An important class of effect algebras was introduced by Gudder in \cite{gudder1} 
and \cite{gudder2}. Fundamental example is the 
standard Hilbert spaces effect algebra ${\mathcal E}({\mathcal H})$.

For an element $x$ of an effect algebra $E$ we denote
%{\phantom{odskok}
%$$
%\begin{array}{r c l l}
%& & & &
%\widetilde{x}&=&\bigvee\{s\in \Sh(E) \mid s\leq x\}&\text{if it exists}
%\end{array}
%$$
$$
\begin{array}{r c  l@{}c l}
\widetilde{x}&=\bigvee_{E}\{s\in \Sh(E) \mid s\leq x\}&%
&\begin{array}{l}
\text{if it exists and}\\
 \text{belongs to}\ \Sh(E),
\end{array}\\
\wwidehat{x}&=\bigwedge_{E}\{s\in \Sh(E) \mid s\geq x\}&%
&\begin{array}{l}%
\text{if it exists and}\\
\text{belongs to}\ \Sh(E).
\end{array}
\end{array}
$$

\begin{definition} {\rm{}\mbox{(\,~\cite{gudder1,gudder2}\,)} 
An effect algebra $(E; \oplus, 0,
1)$ is called {\em sharply dominating} if for every $x\in E$ there
exists $\wwidehat{x}$, the smallest sharp element  such that $x\leq
\wwidehat{x}$. That is $\wwidehat{x}\in \Sh(E)$ and if $y\in \Sh(E)$ satisfies
$x\leq y$ then $\wwidehat{x}\leq y$. }
\end{definition}

Recall that evidently an effect algebra $E$ is sharply dominating iff 
for every $x\in E$ there exists $\widetilde{x}\in \Sh(E)$ such
that $\widetilde{x}\leq x$ and if $u\in \Sh(E)$
satisfies $u\leq x$ then $u\leq \widetilde{x}$.

In what follows set (see \cite{jenca,wujunde})
$$\begin{array}{r@{}l}
\Mea(E)=\{x\in E \mid\ &\text{if}\ v\in \Sh(E)\ \text{satisfies}\\ &v\leq x\ 
\text{then}\ v=0\}.
\end{array}$$

An element $x\in \Mea(E)$ is called {\em meager}. Moreover, $x\in \Mea(E)$ 
iff $\widetilde{x}=0$. Recall that $x\in \Mea(E)$, $y\in E$, $y\leq x$ implies 
$y\in \Mea(E)$ and $x\ominus y\in \Mea(E)$. Moreover, we have a map 
$h:\Sh(E)\to 2^{\Mea(E)}$ that is given by 
$h(s)=\{x\in \Mea(E) \mid x\leq s\}$. 

Recall that $\Mea(E)$ equipped with a partial 
operation $\oplus_{\Mea(E)}$ which is defined, for all $x, y\in \Mea(E)$, 
by $x\oplus_{\Mea(E)} y$ exists
if and only if $x\oplus_E y$  exists and $x\oplus_E y\in \Mea(E)$ in which 
case $x\oplus_{\Mea(E)} y=x\oplus_E y$ is a generalized  effect algebra.

{\renewcommand{\labelenumi}{{\normalfont  (\roman{enumi})}}
\begin{statement}\label{jmpy2} {\rm\cite[Lemma 2.4]{niepa}} Let 
$E$ be an effect algebra 
in which $\Sh(E)$ is a sub-effect algebra of $E$ and 
let $x\in \Mea(E)$ such that $\wwidehat{x}$  exists. Then 
\begin{enumerate}
\settowidth{\leftmargin}{(iiiii)}
\settowidth{\labelwidth}{(iii)}
\settowidth{\itemindent}{(ii)}
\item $\wwidehat{x}\ominus x\in \Mea(E)$.
\item If $y\in \Mea(E)$ such that $x\oplus y$ %
exists and $x\oplus y=z\in\Sh(E)$ then  $\wwidehat{x}=z$. 
\end{enumerate}
\end{statement}}

{\renewcommand{\labelenumi}{{\normalfont  (\roman{enumi})}}
\begin{statement}\label{jpy2} {\rm\cite[Lemma 2.5]{niepa}} Let $E$ be an effect algebra 
in which $\Sh(E)$ is a sub-effect algebra of $E$ and 
let $x\in E$ such that $\widetilde{x}$  % and $\wwidehat{x}$ 
exists. Then 
%\begin{enumerate}
%\settowidth{\leftmargin}{(iiiii)}
%\settowidth{\labelwidth}{(iii)}
%\settowidth{\itemindent}{(ii)}
%\item
${x\ominus \widetilde{x}}\in  \Mea(E)$ and 
$x=\widetilde{x}\oplus {(x\ominus \widetilde{x})}$ is 
the unique decomposition $x = x_S \oplus x_M$, 
where $x_S\in\Sh(E)$ and $x_M \in \Mea(E)$. Moreover, 
$x_S\wedge x_M=0$ and if $E$ is a lattice effect algebra then 
 $x = x_S \vee x_M$. 
%\item If $E$ is a lattice effect algebra such that $\wwidehat{x}$ exists then 
%$\wwidehat{x\ominus \widetilde{x}}$ and  $\wwidehat{\wwidehat{x}\ominus {x}}$ exist,  
%$\wwidehat{x}\ominus \widetilde{x}=\wwidehat{x\ominus \widetilde{x}}=\wwidehat{\wwidehat{x}\ominus {x}}$,
%$\wwidehat{x}=\widetilde{x}\oplus \wwidehat{x\ominus \widetilde{x}}=%
%\widetilde{x}\vee \wwidehat{x\ominus \widetilde{x}}$  and 
%$\widetilde{x}\wedge \wwidehat{x\ominus \widetilde{x}}=0$. 
%Moreover, $\wwidehat{x}\ominus \widetilde{x}=\wwidehat{(\wwidehat{x}\ominus \widetilde{x})\ominus %
%(\wwidehat{x}\ominus {x})}$.
%\end{enumerate}
\end{statement}}

As proved in \cite{cattaneo}, 
$\Sh(E)$ is always a sub-effect algebra in 
a sharply dominating  effect algebra $E$.

\begin{comment}
\begin{corollary}\label{gejza}{\rm{}\cite[Proposition 15]{jenca}}
Let $E$ be a sharply dominating  effect algebra. 
Then every $x \in E$ has a
unique decomposition $x = x_S \oplus x_M$, where $x_S\in\Sh(E)$ and $x_M \in \Mea(E)$, 
namely $x=\widetilde{x}\oplus {(x\ominus \widetilde{x})}$.
\end{corollary}
\end{comment}

\begin{statement}\label{honzazasum}{\rm\cite[Corollary 14]{jenca}, 
\cite[Lemma 4.2]{niepa3}} Let $E$ be an orthocomplete homogeneous effect algebra.
\nopagebreak
\begin{enumerate}
\item[{\rm (i)}]   $E$ is sharply dominating and, for every block $B$ of $E$, $x \in B$ implies 
that $[\widetilde{x}, x]\subseteq B$.
\itrm{(ii)}  Let $x\in \Mea(E)$. 
Then $y=\wwidehat{x}\ominus x$ is the only element such that
\begin{enumerate}
\settowidth{\leftmargin}{(iiiii)}
\settowidth{\labelwidth}{(iii)}
\settowidth{\itemindent}{(ii)}
\item $y\in \Mea(E)$ such that $\wwidehat{y}=\wwidehat{x}$. 
\item $x \oplus_{\Mea(E)} (y\ominus_{\Mea(E)} (x\wedge y))$ 
exists and 
$x \oplus_{\Mea(E)} (y\ominus_{\Mea(E)} (x\wedge y))\in %
h(\wwidehat{x})$.
\item For all $z \in h(\wwidehat{x})$, 
$z \oplus_{\Mea(E)} x \in h(\wwidehat{x})$
if and only if $z \leq y$ and 
$\wwidehat{y\ominus_{\Mea(E)} z}=\wwidehat{x}$.
\end{enumerate}  
\end{enumerate}
\end{statement}

%{cduya}

\section{Triple Representation Theorem for 
TRT-effect algebras}\label{tretikapitola}

In what follows $E$ will be always a homogeneous
sharply dominating    
effect algebra such that, for every block $B$ of $E$, $x \in B$ implies 
that $[\widetilde{x}, x]\subseteq B$ and, 
for all $x\in \Mea(E)$ the element 
$y=\wwidehat{x}\ominus x$ is the only element such that
\begin{enumerate}
\settowidth{\leftmargin}{(iiiii)}
\settowidth{\labelwidth}{(iii)}
\settowidth{\itemindent}{(ii)}
\item $y\in \Mea(E)$ such that $\wwidehat{y}=\wwidehat{x}$. 
\item $x \oplus_{\Mea(E)} (y\ominus_{\Mea(E)} (x\wedge y))$ 
exists and 
$x \oplus_{\Mea(E)} (y\ominus_{\Mea(E)} (x\wedge y))\in %
h(\wwidehat{x})$.
\item For all $z \in h(\wwidehat{x})$, 
$z \oplus_{\Mea(E)} x \in h(\wwidehat{x})$
if and only if $z \leq y$ and 
$\wwidehat{y\ominus_{\Mea(E)} z}=\wwidehat{x}$.
\end{enumerate}
We will call such an effect algebra a {\em TRT-effect algebra}.  
In this case $\Sh(E)$ is a sub-effect algebra 
of $E$.
\medskip

Since $E$ is sharply dominating %%and $\Sh(E)$ is bifull in $E$ 
we have that, for all $x\in \Mea(E)$, 
$$\begin{array}{r@{}l}
\wwidehat{x}&=\bigwedge_{E}\{s\in \Sh(E) \mid x\in h(s)\}\\
&=% 
\bigwedge_{\Sh(E)}\{s\in \Sh(E) \mid x\in h(s)\}.
\end{array}$$ 

This gives us 

\medskip

\begin{tabular}{@{}c@{}l}
&\begin{minipage}{0.479454\textwidth}
\begin{itemize}
\item[(M1)] The mapping\quad $\wwidehat{\phantom{x}}:\Mea(E) \to \Sh(E)$.
\end{itemize}
\end{minipage}
\end{tabular}

\medskip

Note that, if $x\in \Mea(E)$, $s\in \Sh(E)$ and $B$ is a
block of $E$ such that $x, s\in B$ then from 
Statement \ref{grfopu} we have that $x\wedge_{B} s\leq x$ exists. 
Assume that $g\in E$, $g\leq s$ and $g\leq x$. Then $g\in B$ and 
 therefore $g\leq x\wedge_{B} s$. It follows that
 $x\wedge_{E} s=x\wedge_{B} s$. We always have the following partial 
map.

\medskip

\begin{tabular}{@{}c@{}l}
&\begin{minipage}{0.4679454\textwidth}
\begin{itemize}
\item[(M2)] For every $s\in \Sh(E)$, a partial mapping 
$\pi_{s}:\Mea(E) \to h(s)$ is given by 
$\pi_{s}(x)=x\wedge_{E} s$ whenever $\pi_{s}(x)$ is defined.
\end{itemize}
\end{minipage}
\end{tabular}
\medskip 

Therefore, if $x\comp s$  then 
$x\wedge_E s$ exists and $\pi_{s}(x)$ is defined.

\medskip

We also have the following map $R$ which is defined 
entirely in terms of the triple.

\medskip

\begin{tabular}{@{}c@{}l}
&\begin{minipage}{0.4679454\textwidth}
\begin{itemize}
\item[(M3)] The mapping\, $R:\Mea(E) \to \Mea(E)$ 
given by $R(x)=\wwidehat{{x}}\ominus_E x$.
\end{itemize}
\end{minipage}
\end{tabular}
\medskip

Let $x, y\in \Mea(E)$. Let us put 
${\mathscr S}(x, y)=\{z\in \Sh(E)\mid %
z=(z\wedge x)\oplus_E (z\wedge y)\}=\{z\in \Sh(E)\mid % 
\pi_z(x) \ \text{and}\ \pi_z(x)\ \text{are defined},  
z=\wwidehat{\pi_z(x)}\ \text{and}\ R(\pi_z(x))=\pi_z(y)\}$.

\medskip

\begin{tabular}{@{}c@{}l}
&\begin{minipage}{0.4679454\textwidth}
\begin{itemize}
\item[(M4)] The partial mapping\, {$S:\Mea(E) \times \Mea(E)\to \Sh(E)$} 
given by $S(x, y)$ is defined if and only if 
the set ${\mathscr S}(x, y)=\{z\in \Sh(E)\mid z\wedge x \ \text{and}\ z\wedge y\ \text{exist}, %
z=(z\wedge x)\oplus_E (z\wedge y)\}$ has a top element 
$z_0\in {\mathscr S}(x, y)$ 
in which case $S(x, y)=z_0$.
\end{itemize}
\end{minipage}
\end{tabular}

\medskip

Whether $S(x, y)$ is defined or not we are able to decide 
in terms of the triple. Since the eventual top element $z_0$ of 
${\mathscr S}(x, y)$ is in $\Sh(E)$ our definition of ${S}(x, y)$ 
is correct.

\medskip

%\begin{theorem}\label{gejzaour}
\noindent{\bfseries Triple Representation Theorem}
 {\em The triple $((\Sh(E),\oplus_{\Sh(E)}),$ $(\Mea(E),\oplus_{\Mea(E)}),$ $h)$
characterizes $E$ up to isomorphism within the class of all 
TRT-effect algebras.}
%\end{theorem}
\medskip

We have to  construct an isomorphic copy of the original effect algebra
$E$ from the triple $(\Sh(E), \Mea(E), h)$.

\begin{lemma}\label{pommeag} Let $E$ be a TRT-effect algebra, $x, y\in \Mea(E)$. Then $x\oplus_{E} y$ exists 
in $E$
iff $S(x, y)$ is defined in terms of the triple 
$(\Sh(E), \Mea(E), h)$ and 
$\bigl(x\ominus_{\Mea(E)} (S(x, y)\wedge x)\bigr)%
\oplus_{\Mea(E)} \bigl(y\ominus_{\Mea(E)} (S(x, y)\wedge y)\bigr)$ 
exists in $\Mea(E)$ such that %
$\bigl(x\ominus_{\Mea(E)} (S(x, y)\wedge x)\bigr)%
\oplus_{\Mea(E)} \bigl(y\ominus_{\Mea(E)} (S(x, y)\wedge y)\bigr)%
\in h(S(x, y)')$. Moreover, in 
that case
$$\begin{array}{r@{}l}
x\oplus_{E} y={S(x, y)}&\oplus_{E} %
\biggl({\bigl(x\ominus_{\Mea(E)} (S(x, y)\wedge x)\bigr)}\\%
&\oplus_{\Mea(E)}%
{\bigl(y\ominus_{\Mea(E)} (S(x, y)\wedge y)\bigr)}\biggr).
\end{array}$$
\end{lemma}
\begin{proof} %The proof goes literally the same way 
%as in \cite[Lemma 3.1]{{niepa}}. Hence we omit it. 
Assume first that $x\oplus_{E} y$ exists 
in $E$ and let us put $z=x\oplus_{E} y$. Since $E$  is 
sharply dominating we have that
$z=z_S\oplus_{E}  z_M$ such that $z_S\in \Sh(E)$ and 
$z_M\in \Mea(E)$. Since $x\comp y$ by Statement 
\ref{gejzasum}, (e) there is  a  block $B$ 
of $E$ such that $x, y, z\in B$. 
Moreover, $E$ being a TRT-effect algebra yields that $[z_S, z]\subseteq B$ and we obtain 
that $z_S,  z_M\in B$. Therefore $z_S\in \C(B)$ and by Statement 
\ref{gejzapulm}, (i) we have that 
 $z_S= z_S \wedge (x\oplus_{E} y)= z_S \wedge (x\oplus_{B} y)= 
 (z_S \wedge_{B} x)\oplus_{B} (z_S \wedge_{B} y)= %
(z_S \wedge x)\oplus_{E} (z_S \wedge y)$. Hence $z_S\in {\mathscr S}(x, y)$. 
Now, assume that $u\in {\mathscr S}(x, y)$. Then 
$u=(u \wedge x)\oplus_{E} (u \wedge y)\leq x\oplus_{E} y$. Since 
$u\in \Sh(E)$ we have that $u\leq z_S$, i.e., $z_S$ is the top 
element of ${\mathscr S}(x, y)$. Moreover, we have 
$$
\begin{array}{@{}r@{\,}l@{\,}l}
z_S&\oplus_E& z_M=x\oplus_{E} y\\
&=&\biggl(\bigl((S(x, y)\wedge x)\bigr)\oplus_{E} %
\bigl(x\ominus_{E} (S(x, y)\wedge x)\bigr)\biggr)\\[0.5cm]
&\oplus_E &\biggl(\bigl((S(x, y)\wedge y)\bigr)\oplus_{E}
 \bigl(y\ominus_{E} (S(x, y)\wedge y)\bigr)\biggr)\\
&=&{S(x, y)}\oplus_{E} %
\biggl(\bigl(x\ominus_{\Mea(E)} (S(x, y)\wedge x)\bigr)\\[0.5cm]
&&\phantom{S(x, y)}\oplus_{E}%
\bigl(y\ominus_{\Mea(E)} (S(x, y)\wedge y)\bigr)\biggr).
 \end{array}
$$
It follows that $z_M=\bigl(x\ominus_{\Mea(E)} (S(x, y)\wedge x)\bigr)%
\oplus_{E} \bigl(y\ominus_{\Mea(E)} (S(x, y)\wedge y)\bigr)$ and evidently 
$z_M\in h(z_S')$.

Conversely, let us assume that $S(x, y)$ is defined in terms of  $(\Sh(E), \Mea(E), h)$,  
$\bigl(x\ominus_{\Mea(E)} (S(x, y)\wedge x)\bigr)%
\oplus_{\Mea(E)} \bigl(y\ominus_{\Mea(E)} (S(x, y)\wedge y)\bigr)$ 
exists in $\Mea(E)$ and %
$\bigl(x\ominus_{\Mea(E)} (S(x, y)\wedge x)\bigr)%
\oplus_{\Mea(E)} \bigl(y\ominus_{\Mea(E)} (S(x, y)\wedge y)\bigr)%
\in h(S(x, y)')$. Then 
$\bigl(x\ominus_{\Mea(E)} (S(x, y)\wedge x)\bigr)%
\oplus_{\Mea(E)} \bigl(y\ominus_{\Mea(E)} (S(x, y)\wedge y)\bigr)%
\leq S(x, y)'$, i.e., 
$$
\begin{array}{r c l}
z&=&S(x, y)\oplus_{E}%
\biggl(\bigl(x\ominus_{\Mea(E)} (S(x, y)\wedge x)\bigr)\\%
&&\phantom{S(x, y)}%
\oplus_{\Mea(E)} \bigl(y\ominus_{\Mea(E)} (S(x, y)\wedge y)\bigr)\biggr)\\
&=&\biggl(\bigl(S(x, y)\wedge x\bigr)\oplus_{E} %
\bigl(S(x, y)\wedge y\bigr)\biggr) \oplus_{E}\\[0.5cm]%
&&\biggl(\bigl(x\ominus_{E} (S(x, y)\wedge x)\bigr)%
\oplus_{E} \\
&&\bigl(y\ominus_{E} (S(x, y)\wedge y)\bigr)\biggr)=x\oplus_{E} y\\
\end{array}
$$
\noindent{}is defined.
\end{proof}

{\renewcommand{\labelenumi}{{\normalfont  (\roman{enumi})}}
\begin{theorem}\label{tripletheor}
Let E be a TRT-effect algebra. 
Let $\Tea(E)$ be a subset
of $\Sh(E) \times \Mea(E)$ given by
$$\Tea(E) =\{(z_S, z_M)\in \Sh(E) \times \Mea(E) \mid z_M \in  h(z_S')\}.$$
Equip $\Tea(E)$ with a partial binary operation $\oplus_{\Tea(E)}$ with 
$(x_S, x_M) \oplus_{\Tea(E)} (y_S, y_M)$ is defined if and  only if 
\begin{enumerate}
\item $S(x_M, y_M)$ is defined,
\item $z_S=x_S\oplus_{\Sh(E)} y_S \oplus_{\Sh(E)} S(x_M, y_M)$ is defined,
\item $z_M=\biggl(x_M\ominus_{\Mea(E)} \bigl(S(x_M, y_M)\wedge x_M\bigr)\biggr)%
\oplus_{\Mea(E)} %
\biggl(y_M\ominus_{\Mea(E)} \bigl(S(x_M, y_M)\wedge y_M\bigr)\biggr)$ is defined,
\item $z_M\in h(z_S')$. 
\end{enumerate}
In this case $(z_S, z_M)=(x_S, x_M) \oplus_{\Tea(E)} (y_S, y_M)$.
Let\/ $0_{\Tea(E)}=(0_E,0_E)$ and\/ $1_{\Tea(E)}=(1_E,0_E)$. Then 
$\Tea(E)=(\Tea(E), \oplus_{\Tea(E)}, 0_{\Tea(E)},$ $1_{\Tea(E)})$ 
is an effect algebra and the mapping 
$\varphi:E\to \Tea(E)$  given by 
$\varphi(x) = (\widetilde{x}, x\ominus_{E} \widetilde{x})$   
is an isomorphism of effect algebras.
\end{theorem}}
\begin{proof} Evidently, $\varphi$ is correctly defined since, 
for any $x\in E$, we have that $x=\widetilde{x}\oplus_{E} (x\ominus \widetilde{x})=%
x_S\oplus_{E} x_M$, $x_S\in \Sh(E)$ and $x_M\in \Mea(E)$. Hence 
$\varphi(x) = (x_S, x_M)\in \Sh(E) \times \Mea(E)$ and $x_M\in h(x_S')$. Let us check that $\varphi$ is 
bijective. Assume first that $x, y\in E$ such that $\varphi(x) =\varphi(y)$. 
We have $x=\widetilde{x}\oplus_{E} (x\ominus_{E} \widetilde{x})=%
\widetilde{y}\oplus_{E} (y\ominus_{E} \widetilde{y})=y$. Hence 
$\varphi$ is injective. Let $(x_S, x_M)\in \Sh(E) \times \Mea(E)$ and 
$x_M\in h(x_S')$. This yields that $x=x_S\oplus_{E} x_M$ exists and 
evidently by Lemma \ref{jpy2}, (i) $\widetilde{x}=x_S$ and $x\ominus_{E} \widetilde{x}=x_M$.
It follows that $\varphi$ is surjective. Moreover, 
$\varphi(0_{E})=(0_E,0_E)=0_{\Tea(E)}$ and\/ $\varphi(1_{E})=(1_E,0_E)=1_{\Tea(E)}$.

Now, let us check that, for all $x, y\in E$, $x\oplus_E y$ 
is defined iff $\varphi(x)\oplus_{\Tea(E)} \varphi(y)$ is defined 
in which case $\varphi(x\oplus_E y)=\varphi(x)\oplus_{\Tea(E)} \varphi(y)$.  
For any $x, y, z\in E$ we obtain 
\begin{center}
\begin{tabular}{@{}c r c l}
\multicolumn{4}{@{}l}{$z=x\oplus_E y$  $\iff$}\\ 
\multicolumn{4}{@{}l}{$z=\bigl(\widetilde{x}%
\oplus_E (x\ominus_{E} \widetilde{x})\bigr)%
\oplus_E \bigl(\widetilde{y}\oplus_E (y\ominus_{E} \widetilde{y})\bigr)$% 
\, $\iff$} \\[0.1cm]
\multicolumn{4}{@{}l}{$z=(\widetilde{x}\oplus_E \widetilde{y})%
\oplus_E \bigl((x\ominus_{E} \widetilde{x})%
\oplus_E (y\ominus_{E} \widetilde{y})\bigr)$\, $\iff$} \ %
\\[0.2cm]
\multicolumn{4}{@{}l}{by Lemma \ref{pommeag}\ $(\exists u\in E)\ %
u=S(x\ominus_{E} \widetilde{x}, y\ominus_{E} \widetilde{y})$}\\
\multicolumn{4}{@{}l}{  and}\\ %
\multicolumn{4}{@{}l}{$\begin{array}{@{}r@{}l}
z=(\widetilde{x}&\oplus_E \widetilde{y})% 
\oplus_E\\[0.2cm]
&\biggl(u\oplus_{E}\bigl((x\ominus_{E} \widetilde{x})\ominus_{E} (u\wedge (x\ominus_{E} \widetilde{x}))\bigr)\oplus_{E}\\
&\bigl((y\ominus_{E} \widetilde{y})\ominus_{E} (u\wedge (y\ominus_{E} \widetilde{y}))\bigr)\biggr)\ 
\end{array}$ }\\
$\iff$&\multicolumn{3}{@{}l}{$(\exists u\in E)\ %
u=S(x\ominus_{E} \widetilde{x}, y\ominus_{E} \widetilde{y})$  and}\\ %
\multicolumn{4}{@{}l}{$\begin{array}{r@{}l}
z=(\widetilde{x}\oplus_E \widetilde{y}&% 
\oplus_E u)\oplus_{E}\\[0.2cm]
&\biggl(\bigl((x\ominus_{E} \widetilde{x})\ominus_{E} (u\wedge (x\ominus_{E} \widetilde{x}))\bigr)\oplus_{E}\\
&\phantom{(}\bigl((y\ominus_{E} \widetilde{y})\ominus_{E} (u\wedge (y\ominus_{E} \widetilde{y}))\bigr)\biggr)\ 
\end{array}$ }\\
%\end{tabular}
%\begin{tabular}{@{}c r c l}
$\iff$&\multicolumn{3}{@{}l}{$(\exists u\in E)\ %
u=S(x\ominus_{E} \widetilde{x}, y\ominus_{E} \widetilde{y})$  and}\\[0.1cm] %
\multicolumn{4}{@{}l}{$\begin{array}{r@{}l}
z=(\widetilde{x}&\oplus_{\Sh(E)} \widetilde{y}% 
\oplus_{\Sh(E)}u)\oplus_{E} \\
&\biggl(\bigl((x\ominus_{E} \widetilde{x})\ominus_{\Mea(E)} %
(u\wedge (x\ominus_{E} \widetilde{x}))\bigr)\oplus_{\Mea(E)}\\
&\phantom{(}\bigl((y\ominus_{E} \widetilde{y})\ominus_{\Mea(E)} %
(u\wedge (y\ominus_{E} \widetilde{y}))\bigr)\biggr)\ 
\end{array}$ }\\
\end{tabular}
\begin{tabular}{@{}c r c l}
$\iff$&\multicolumn{3}{@{}l}{$(\widetilde{x}, x\ominus_{E} \widetilde{x}) %
\oplus_{\Tea(E)} (\widetilde{y}, y\ominus_{E} \widetilde{y})$ is defined}\\
\multicolumn{4}{@{}l}{  and}\\ %
\multicolumn{4}{@{}l}{%
$\begin{array}{@{}r@{}c@{}l@{}l}\varphi(z)&=&%
\biggl(&\widetilde{x}\oplus_{\Sh(E)}\widetilde{y}\oplus_{\Sh(E)}\\%
&&&S(x\ominus_{E} \widetilde{x}, y\ominus_{E} \widetilde{y}), %
\bigl((x\ominus_{E} \widetilde{x}) \ominus_{\Mea(E)} \\%
&&&(S(x\ominus_{E} \widetilde{x}, y\ominus_{E} \widetilde{y})%
{\wedge} (x\ominus_{E} \widetilde{x}))\bigr)\oplus_{\Mea(E)}\\ 
&&&\bigl((y\ominus_{E} \widetilde{y})\ominus_{\Mea(E)} \\%
&&&%
(S(x\ominus_{E} \widetilde{x}, y\ominus_{E} \widetilde{y})%
\wedge (y\ominus_{E} \widetilde{y}))\bigr)\biggr)\\
&=&(&\widetilde{x}, x\ominus_{E} \widetilde{x}) %
\oplus_{\Tea(E)} (\widetilde{y}, y\ominus_{E} \widetilde{y})\\
&=& &%
\varphi(x)\oplus_{\Tea(E)} \varphi(y).
\end{array}$}
\end{tabular}
\end{center}

Altogether, $\Tea(E)=(\Tea(E), \oplus_{\Tea(E)}, 0_{\Tea(E)},$ $1_{\Tea(E)})$ 
is an effect algebra and the mapping $\varphi:E\to \Tea(E)$     
is an isomorphism of effect algebras.
\end{proof}

The Triple Representation Theorem then follows immediately.

\begin{remark}\label{larem}{\rm First, note that 
any sharply dominating Archimedean 
atomic  lattice effect algebra is a TRT-effect algebra. 
Namely, any  lattice effect algebra is homogeneous. 
From \cite[Theorem 2.10, (iv)]{niepa} we get that 
$[\widetilde{x}, x]\subseteq B$, for every block $B$ of $E$ and 
every $x \in B$. The remaining condition follows from 
\cite[Theorem 2.10, (v)]{niepa}.\\
Second, any homogeneous orthocomplete effect algebra is a TRT-effect algebra in virtue of Statement \ref{honzazasum}. This 
immediately yields that any complete lattice effect algebra 
is  a TRT-effect algebra.}
\end{remark}

\section*{Acknowledgements}  J. Paseka gratefully acknowledges Financial Support 
of the  Ministry of Education of the Czech Republic
under the project MSM0021622409 and of Masaryk University under the grant 0964/2009. 
Both authors acknowledge the support by ESF Project CZ.1.07/2.3.00/20.0051
Algebraic methods in Quantum Logic of the Masaryk University.

\end{document}